\documentclass[11pt,a4paper]{amsart}

\usepackage[active]{srcltx}
\usepackage[T1]{fontenc}
\usepackage[latin1]{inputenc}
\usepackage{hyperref}
\usepackage{enumitem}
\usepackage[all]{xy}
\usepackage{amssymb}
\usepackage[mathcal]{euscript}
\usepackage[svgnames]{xcolor}
\usepackage{ae,aecompl}
\usepackage{multirow}
\usepackage{makecell}

\setenumerate[1]{label=$(\arabic*)$,leftmargin=*}
\setenumerate[2]{label=$(\alph*)$,leftmargin=*}
\setenumerate[3]{label=$(\roman*)$,leftmargin=*}
\setitemize[1]{label=$-$,leftmargin=*}
\setitemize[2]{label=$\circ$,leftmargin=*}
\setitemize[3]{label=$--$,leftmargin=*}
\setlist[1]{leftmargin=*}

\newtheorem{Thm}{Theorem}[section]
\newtheorem{Prop}[Thm]{Proposition}
\newtheorem{Lemma}[Thm]{Lemma}
\newtheorem{Cor}[Thm]{Corollary}

\theoremstyle{definition}
\newtheorem{Def}[Thm]{Definition}

\numberwithin{equation}{section}

\newdir^{((}{{}*!/-5pt/\dir^{(}}
\def\op{\lbrack\hskip-1.6pt\lbrack\relax}
\def\cl{\rbrack\hskip-1.6pt\rbrack\relax}
\def\pv#1{\ensuremath{{\sf#1}}}
\def\Cl#1{\ensuremath{\mathcal #1}}
\def\Om#1#2{\ensuremath{\overline\Omega_{#1}{\sf#2}}}
\def\om#1#2{\ensuremath{\Omega_{#1}{\sf#2}}}
\def\omo#1#2{\ensuremath{\Omega^{\omega}_{#1}{\sf#2}}}
\def\oms#1#2{\ensuremath{\Omega^{\sigma}_{#1}{\sf#2}}}
\def\pj#1{\ensuremath{p_{\sf#1}}}
\def\tclcv#1#2{\ensuremath{\mathrm{cl}_{\omega,\pv{#1}}(#2)}}
\def\tclc#1{\tclcv{S}{#1}}

\begin{document}

\author[J. Almeida]{Jorge Almeida}
\address[J. Almeida]{CMUP, Departamento de Matem\'atica,
  Faculdade de Ci\^encias\\
  Universidade do Porto, Rua do Campo Alegre, 687, 4169-007 Porto,
  Portugal.}
\email{\href{mailto:jalmeida@fc.up.pt}{jalmeida@fc.up.pt}}

\author[J. C. Costa]{Jos\'e Carlos Costa}
\address[J. C. Costa]{Centro de Matem\'atica, Universidade do Minho,
  Campus de Gualtar, 4700-320 Braga, Portugal.}
\email{\href{mailto:jcosta@math.uminho.pt}{jcosta@math.uminho.pt}}

\author[M. Zeitoun]{Marc Zeitoun}
\address[M. Zeitoun]{LaBRI, Universit\'es de Bordeaux \& CNRS UMR~5800.
  351 cours de la Lib\'eration, 33405 Talence Cedex, France.}
\email{\href{mailto:mz@labri.fr}{mz@labri.fr}}

\title{Recognizing pro-\pv R closures of regular languages}

\date{\today}

\begin{abstract}
  Given a regular language $L$, we effectively construct a unary
  semigroup that recognizes the topological closure of~$L$ in the free
  unary semigroup relative to the variety of unary semigroups generated
  by the pseudovariety \pv R of all finite \Cl R-trivial semigroups. In
  particular, we obtain a new effective solution of the separation
  problem of regular languages by \pv R-languages.
\end{abstract}

\keywords{Free profinite semigroup, unary semigroup, regular language,
  profinite closure, R-trivial semigroup, algebraic recognition,
  omega-term}

\subjclass[2010]{Primary: 20M07; Secondary:
  20M35, 22A99}

\maketitle

\section{Introduction}
\label{sec:intro}

There is a remarkable connection between the theories of finite
semigroups and regular languages. At its basis is the well known and
simple fact of the finiteness of the syntactic semigroups of such
languages, which may be effectively computed as the transition
semigroups of their minimal automata. This suggests a method for testing
whether a regular language has a certain combinatorial property, namely
by verifying whether its syntactic semigroup enjoys an associated
algebraic property. A general framework for this kind of problems and a
characterization of which properties may be handled in this way was
given by Eilenberg~\cite{Eilenberg:1976}. On the semigroup side, the
relevant algebraic properties define so-called pseudovarieties, which
are nonempty classes of finite semigroups closed under taking
homomorphic images, subsemigroups and finite direct products. In
particular, Eilenberg's result prompted considerable interest in
studying pseudovarieties.

For the above method to be successful for a suitable combinatorial
property, one needs to be able to test membership of a given finite
semigroup in the corresponding pseudovariety. Thus, a key question on a
pseudovariety is to determine whether its membership problem admits an
algorithmic solution, in which case the pseudovariety is said to be
decidable. It turns out that several combinatorial constructions on
classes of languages correspond to operations on pseudovarieties that
are known not to preserve decidability in general
\cite{Albert&Baldinger&Rhodes:1992,Auinger&Steinberg:2001b}. This fact
has led to the search for stronger algorithmic properties that may be
preserved by such operations. The notion of a tame pseudovariety, in its
various flavors, has emerged from this approach
\cite{Almeida&Steinberg:2000a}, inspired by seminal work of
Ash~\cite{Ash:1991}. A quick introduction to this line of ideas and its
applications may be found in~\cite{Almeida:2003cshort}.

Tameness is intimately connected with profinite topologies. Roughly
speaking, tameness of a pseudovariety \pv V means that there is a
natural algebraic structure on profinite semigroups, with the same
homomorphisms, enjoying special properties. Profinite semigroups are
then naturally viewed as algebras of that kind and one may speak of the
variety of such algebras generated by~\pv V. One of the key properties
is the word problem in such relatively free algebras. The other key
property has to do with the solution, modulo~\pv V, of finite systems of
semigroup equations with clopen constraints: should the system admit a
solution, does it also have a solution in the restricted algebraic
language?

Even rather simple systems of equations as that reduced to the single
equation $x=y$ lead to highly nontrivial problems on pseudovarieties of
interest. Determining whether that equation with clopen constraints has
a solution modulo~\pv V is equivalent to the following \pv V-separation
problem: given two regular languages, determine whether there is a
language whose syntactic semigroup belongs to~\pv V which contains one
of them and is disjoint from the other; in topological terms,
this means that the closures in the free pro-V semigroup of the given
languages are disjoint \cite{Almeida:1996d}. The algorithmic solution of
this problem for various pseudovarieties turns out to have numerous
applications (see, for instance, \cite{Place&Zeitoun:2016a}).

Among pseudovarieties that have deserved a lot of attention, for their
connections with formal language theory or for their inherent algebraic
interest, is the pseudovariety \pv R of all finite \Cl R-trivial
semigroups, that is finite semigroups in which every principal right
ideal admits only one element as a generator. Its word problem for the
signature consisting of multiplication and the $\omega$-power (which, in
a finite semigroup, gives the idempotent power of the base) has a
particularly nice solution~\cite{Almeida&Zeitoun:2003b} (see
also~\cite{Kufleitner&Wachter:2018}). Moreover, \pv R has very strong
tameness properties~\cite{Almeida&Costa&Zeitoun:2005b} with respect to
this signature.

The main contribution of this paper is to show that the pro-\pv R
closure of a regular language in the free $\omega$-semigroup relatively
to the pseudovariety \pv R is recognized by a homomorphism into a finite
$\omega$-semigroup. The proof is constructive: starting from a finite
automaton recognizing the given regular language, we construct a finite
recognizer for the pro-\pv R closure of the language, in which the image
of the language is effectively computable. As a consequence, we obtain a
new algorithm to test whether the intersection of the pro-\pv R closures
of finitely many regular languages is empty or not. Indeed, this
property is clearly decidable given finite recognizers for these
closures. This problem is known to be equivalent to testing whether a
subset of a finite semigroup is \pv R-pointlike \cite{Almeida:1996d}.
Therefore, our result provides an algorithmic solution for it, and also
for testing whether such a subset is \pv R-idempotent pointlike. In
particular, we solve the \pv R-separation problem for regular languages.

The paper is organized as follows. In Section~\ref{sec:preliminaries},
we introduce the necessary terminology and background.
Section~\ref{sec:R-closures} serves to construct a first finite
approximation to a semigroup modeling the $\omega$-words in the closure
of a regular language. A suitable (but unnatural) $\omega$-power and a
natural partial order on such a finite semigroup are considered
respectively in Sections~\ref{sec:alt-w-power} and~\ref{sec:order}.
Finally, in Section~\ref{sec:recognition}, it is shown that the
previously constructed unary semigroup recognizes the topological
closure of the given regular language, and some decidability
applications are drawn.

\section{Preliminaries}
\label{sec:preliminaries}

The reader is referred to \cite{Almeida:2003cshort,Almeida&Costa:2015hb}
for quick introductions to the topics of this paper. Nevertheless, we
briefly recall the notions involved in our discussions.

Finite semigroups are viewed as discrete topological spaces. A
\emph{profinite semigroup} is an inverse limit of an inverse system of
finite semigroups; equivalently, it is a (multiplicative) semigroup with
a topology for which the multiplication is continuous and such that the
topology is compact (Hausdorff) and zero-dimensional. Given an element
$s$ of a profinite semigroup and an integer $k$, the sequence
$(s^{n!+k})_n$ converges to an element, denoted $s^{\omega+k}$. In
particular, for $k=0$, we get the element $s^\omega=s^{\omega+0}$, which
is idempotent.

By a \emph{pseudovariety} we mean a (nonempty) class \pv V of finite
semigroups that is closed under taking homomorphic images, subsemigroups
and finite direct products. The pseudovariety of all finite semigroups
is denoted \pv S. A~profinite semigroup $S$ is said to be \emph{pro-\pv
  V} if distinct points may be separated by continuous homomorphisms
into semigroups from~\pv V. For a finite set $A$, a pro-\pv V semigroup
$S$ is said to be \emph{free pro-\pv V over $A$} if there is a mapping
$\iota:A\to S$ whose image generates a dense subsemigroup of~$S$ and
such that, for every function $\varphi:A\to T$ into a pro-\pv V
semigroup $T$, there is a unique continuous homomorphism
$\hat{\varphi}:S\to T$ such that $\hat{\varphi}\circ\iota=\varphi$. Such
a pro-\pv V semigroup $S$ always exists and it is clearly unique up to
homeomorphic isomorphism. It is denoted \Om AV. The elements of \Om AV
are called \emph{pseudowords over~\pv V} or simply \emph{pseudowords} if
$\pv V=\pv S$. The unique continuous homomorphism $\Om AS\to\Om AV$
induced by the generating mapping $A\to\Om AV$ is denoted $\pj V$.

Consider the pseudovariety \pv{Sl} of all finite semilattices,
commutative semigroups in which all elements are idempotents. As is well
known, we may view \Om A{Sl} as the semigroup of nonempty subsets of~$A$
under the operation of union. The continuous homomorphism $\pj{Sl}:\Om
AS\to\Om A{Sl}$ that sends each free generator $a\in A$ to $\{a\}$ is
also denoted $c$ and it is called the \emph{content} function.

A key pseudovariety in our study is the class \pv R of all finite
semigroups in which Green's relation \Cl R is trivial, that is, if two
elements generate the same right ideal then they are equal.

The \emph{cumulative content} $\vec c(w)$ of a pseudoword $w\in\Om AS$
consists of all letters $a\in A$ for which there exists a factorization
$w=uv$ with $\pj R(v)$ idempotent and $a\in c(v)$. The terminology comes
from~\cite{Almeida&Weil:1996b}, where it was used in a more restrictive
sense, and~\cite{Almeida&Zeitoun:2003b}, where the definition is easily
recognized to be equivalent to the one adopted here.

This paper deals specially with unary semigroups, that is semigroups
with an additional unary operation, which will be usually denoted as an
$\omega$-power. Such unary semigroups will, therefore, often be called
\emph{$\omega$-semigroups}. As we have observed above, profinite
semigroups have a natural structure of $\omega$-semigroups. In
particular, we may consider the variety of $\omega$-semigroups generated
by the semigroups of a given pseudovariety \pv V; it is denoted $\pv
V^\omega$. The $\omega$-semigroup in~$\pv V^\omega$ freely generated by
a (finite) set $A$ may be obtained as the $\omega$-subsemigroup of~\Om
AV generated by $A$; it is denoted \omo AV. Elements of the free
$\omega$-semigroup are called (semigroup) \emph{$\omega$-terms}.

The generating mapping $\iota:A\to\Om AV$ extends uniquely to a
homomorphism $A^+\to\Om AV$ defined on the semigroup $A^+$ freely
generated by~$A$; it is also denoted $\iota$. For a language $L\subseteq
A^+$, we denote $\tclcv VL$ the topological closure of $\iota(L)$ in the
subspace \omo AV. A property of a pseudovariety \pv V introduced
in~\cite{Almeida&Steinberg:2000a} that plays an important role is that
of being \emph{$\omega$-full}. We take as the definition the equivalent
formulation given
in~\cite[Proposition~4.3]{Almeida&Costa&Zeitoun:2014}: a pseudovariety
\pv V is $\omega$-full if and only if the equality $\pj V(\tclcv SL) =
\tclcv VL$ holds for every regular language $L\subseteq A^+$.

The main result of this paper (Theorem~\ref{t:R-closures}) is that
\tclcv RL is recognized by a homomorphism onto an effectively
constructible finite $\omega$-semigroup. In contrast, it should be noted
that the analogous result does not hold for the pseudovariety \pv G of
all finite groups. Indeed, the variety of $\omega$-semigroups $\pv
G^\omega$ satisfies the identities $x^\omega y=y=yx^\omega$, which
forces the interpretation of the $\omega$-power in finite members
of~$\pv G^\omega$ to be the natural one. It follows that the members
of~$\pv G^\omega$ are finite groups and the subsets of $\omo
AG=A^+\cup\{1\}$ recognized by homomorphisms into members of~$\pv
G^\omega$ are the closures of \pv G-languages. Thus, members of~$\pv
G^\omega$ cannot recognize the closures in~\omo AG of arbitrary regular
languages of~$A^+$.

Similar considerations apply to the language in which the $\omega$-power
is replaced by the $(\omega-1)$-power, which is more suitable to capture
group phenomena. It is not excluded though that there is some even
richer language that will be sufficient to obtain a result similar to
our main theorem for the pseudovariety of groups. A somewhat related
phenomenon is that \pv G is not tame for the language of
$(\omega-1)$-semigroups for arbitrary finite systems of word
equations~\cite{Coulbois&Khelif:1999}. The quest for a richer language
capturing tameness is also open. We do not know if there is a connection
between the two properties, namely recognition of closures of regular
languages in \oms AV by homomorphisms into finite $\sigma$-algebras and
$\sigma$-tameness with respect to arbitrary finite systems of equations.

\section{A semigroup modeled after \texorpdfstring{\pv R}{R}}
\label{sec:R-closures}

We introduce in this section a finite semigroup which is meant to
capture certain parameters of pseudowords over~\pv R. The precise
connection is delayed until Section~\ref{sec:order}, where it plays an
important role.

Let $A$ be a finite alphabet. Consider the following pseudovarieties of
bands:
\begin{alignat*}{4}
  \pv{LRB}&=\op xyx=xy,\ x^2=x\cl&&(\mbox{left regular bands})\\
  \pv{MNB}&=\op xyxzx=xyzx,\ x^2=x\cl&\quad&(\mbox{regular bands}).
\end{alignat*}
Note that the solution of the word problem in the relatively free
semigroup \om A{LRB}, that is the identity problem for~\pv{LRB}, is
obtained by reducing each word $w$ to the canonical form which retains
from $w$ only the leftmost occurrence of each letter.

In the following result, we consider a first approximation to behavior
of pseudowords over~\pv R. This is done taking a pair where the first
component models the cumulative content, while the second registers the
order of the first occurrences of letters.

\begin{Lemma}
  \label{l:cumul-cont,LRB}
  Let $\Cl L_A$ be the subset of the Cartesian product $\Cl
  P(A)\times(\om A{LRB})^1$ consisting of all pairs $(B,u)$ such that
  $B\subseteq c(u)$. For $(B,u)$ and $(C,v)$ in~$\Cl L_A$, let
  $(B,u)(C,v)=(D,uv)$, where
  $$D=
  \begin{cases}
    B&\mbox{if } c(v)\subseteq B\\
    C&\mbox{otherwise.}
  \end{cases}
  $$
  This defines an associative operation on $\Cl L_A$ which turns it into
  a band.
\end{Lemma}

\begin{proof}
  We first check that the operation is associative. Consider three
  elements $(B,u)$, $(C,v)$, and $(D,w)$ of~$\Cl L_A$. We verify that
  \begin{equation}
    \label{eq:cum-cont+cont-1}
    (B,u)(C,v)\cdot(D,w)=(B,u)\cdot(C,v)(D,w).
  \end{equation}
  If $c(vw)\subseteq B$, then $(B,u)(C,v)=(B,uv)$, and so
  $(B,u)(C,v)\cdot(D,w)=(B,uvw)$, while
  $(B,u)\cdot(C,v)(D,w)=(B,u)(E,vw)=(B,uvw)$, independently of the value
  of~$E$. In the remaining cases, namely when $c(vw)\nsubseteq B$,
  one of the following must hold:
  \begin{enumerate}[label=$(\roman*)$,ref=$(\roman*)$,leftmargin=*]
  \item\label{item:cumul-cont,LRB-1} $c(v)\subseteq B$;
  \item\label{item:cumul-cont,LRB-2} $c(v)\nsubseteq B$ and $c(w)\nsubseteq C$;
  \item\label{item:cumul-cont,LRB-3} $c(v)\nsubseteq B$ and
    $c(w)\subseteq C$;
  \end{enumerate}
  In case \ref{item:cumul-cont,LRB-1}, we have $c(w)\nsubseteq B$ and,
  since $C\subseteq c(v)$, we conclude that $c(w)\nsubseteq C$, which
  entails that both sides of~\eqref{eq:cum-cont+cont-1} give $(D,uvw)$.
  In case \ref{item:cumul-cont,LRB-2}, both sides
  of~\eqref{eq:cum-cont+cont-1} also give $(D,uvw)$ while in case
  \ref{item:cumul-cont,LRB-3}, they both give $(C,uvw)$.

  It is immediate that every element of $\Cl L_A$ is idempotent, so that
  $\Cl L_A$ is a band.
\end{proof}

One may ask how low $\Cl L_A$ falls in the lattice of pseudovarieties of
bands. It is not hard to show that, whenever $|A|\ge2$, the
pseudovariety generated by $\Cl L_A$ is precisely $\pv{MNB}$.

For a word $v\in A^*$ and a subset $B$ of~$A$, we let
$\mathrm{i}_B(v)$ denote the leftmost letter of~$v$ that does not belong
to~$B$, if such a letter exists, or else the empty word.

For a finite alphabet $A$, let $A^1=A\uplus\{1\}$.
Given two elements $(B,u)$ and $(C,v)$ of $\Cl L_A$, we define a
function
\begin{align*}
  \chi^B_{u,v}:A^1&\to A^1\times A^1 \\
  a &\mapsto
      \begin{cases}
        (a,\mathrm{i}_B(v)) & \text{if } a\in c(u) \vee c(v)\subseteq B \\
        (1,a) & \text{if } a\in A^1\setminus c(u) \wedge c(v)\nsubseteq B.
      \end{cases}
\end{align*}

Given two functions $f:X\to X^m$ and $g:X\to X^n$ with respective
components $f_i$ ($i=1,\ldots,m$) and $g_j$ ($j=1,\ldots,n$), let
$(f,g):X\times X\to X^{m+n}$ be defined by the formula
$$(f,g)(x,y)=(f_1(x),\ldots,f_m(x),g_1(y),\ldots,g_n(y)).$$
Further, let $\mathrm{Id}_X$ denote the identity function on the set
$X$.

\begin{Lemma}\label{l:over-associativity}
  Let $(B,u)$, $(C,v)$, and $(D,w)$ be elements of~$\Cl L_A$. Then, the
  following equality holds, where $(X,uv)=(B,u)(C,v)$:
  \begin{equation}
    \label{eq:chi-associativity}
    (\chi^B_{u,v},\mathrm{Id}_{A^1})\circ\chi^X_{uv,w}
    =(\mathrm{Id}_{A^1},\chi^C_{v,w})\circ\chi^B_{u,vw}.
  \end{equation}
\end{Lemma}

\begin{proof}
  Both sides of Equation~\eqref{eq:chi-associativity} are functions
  $A^1\to A^1\times A^1\times A^1$. We show that they coincide on each
  $a\in A^1$. Consider the following function values:
  \begin{alignat*}{3}
    (x,a_3)&=\chi^X_{uv,w}(a) &\qquad\quad&
    (a_1,a_2)&=\chi^B_{u,v}(x) \\
    (b_1,y)&=\chi^B_{u,vw}(a) &&
    (b_2,b_3)&=\chi^C_{v,w}(y).
  \end{alignat*}
  We verify that $(a_1,a_2,a_3)=(b_1,b_2,b_3)$. This is a somewhat
  tedious case-by-case calculation which is summarized in the following
  table.

{\footnotesize%
\setlength{\tabcolsep}{1.3pt}%
\renewcommand{\arraystretch}{1.7}%
$$
\begin{tabular}{|c!{\vrule width1pt}c|c|c|c|c|c|}
  \hline
  & \multicolumn{2}{c|}{$c(v)\subseteq B$}
  & \multicolumn{3}{c|}{$c(v)\nsubseteq B$}
  & $c(vw)\nsubseteq B$
  \\ \cline{2-7}
  & \multirow{2}{*}{$c(w)\subseteq B$}
  & $c(w)\nsubseteq B$
  & \multirow{2}{*}{$a\in c(u)$}
  & $c(w)\subseteq C$
  & \multicolumn{2}{c|}{$c(w)\nsubseteq C$}
  \\ \cline{3-3}\cline{5-7}
  &
  & $a\in c(u)$
  & 
  & $a\notin c(u)$
  & $a\in c(v)\setminus c(u)$
  & $a\notin c(uv)$
  \\ \Xcline{1-7}{1pt}
  $X$
  & \multicolumn{2}{c|}{$B$}
  & \multicolumn{3}{c|}{$C$}
  &
  \\ \cline{1-7}
  $(x,a_3)$
  & \multicolumn{2}{c|}{$(a,\mathrm{i}_B(w))$}
  & \multicolumn{3}{c|}{($a,\mathrm{i}_C(w))$}
  & $(1,a)$
  \\ \cline{1-7}
  $(a_1,a_2)$
  & \multicolumn{3}{c|}{$(a,\mathrm{i}_B(v))$}
  & \multicolumn{2}{c|}{$(1,a)$}
  & $(1,1)$
  \\ \cline{1-7}
  $(b_1,y)$
  & \multicolumn{3}{c|}{$(a,\mathrm{i}_B(vw))$}
  & \multicolumn{3}{c|}{$(1,a)$}
  \\ \cline{1-7}
  $(b_2,b_3)$
  & $(1,1)$
  & $(1,\mathrm{i}_B(w))$
  & $(\mathrm{i}_B(v),\mathrm{i}_C(w))$
  & \multicolumn{2}{c|}{$(a,\mathrm{i}_C(w))$}
  & $(1,a)$
  \\ \Xcline{1-7}{1pt}
  $(a_1,a_2,a_3)$
  & \multicolumn{2}{c|}{\multirow{2}{*}{$(a,1,\mathrm{i}_B(w))$}}
  & \multirow{2}{*}{$(a,\mathrm{i}_B(v),\mathrm{i}_C(w))$}
  & \multicolumn{2}{c|}{\multirow{2}{*}{$(1,a,\mathrm{i}_C(w))$}}
  & \multirow{2}{*}{$(1,1,a)$}
  \\ \cline{1-1}
  $(b_1,b_2,b_3)$
  & \multicolumn{2}{c|}{}
  &
  & \multicolumn{2}{c|}{}
  &
  \\ \cline{1-7}
\end{tabular}
$$}

  \noindent
  The conditions in each column in the top part of the table define a
  partition of all relevant cases and, with only one exception, where
  the value of $X$ remains undetermined, are sufficient to determine the
  values corresponding to the entries in the first column in the
  remainder of the table. Those values are obtained by simply applying
  the definition of the $\chi$ functions. In the last column, although,
  except for the value of~$X$, the remaining values obtained do not
  depend on whether or not $c(v)$ is contained in~$B$, it is useful to
  distinguish the two cases in the calculation. We leave it to the
  reader to check that all the values are correct.
\end{proof}

Consider a finite set $Q$. Let $\Cl B(Q)$ be the monoid of all binary
relations on~$Q$.

Given two functions $F,G\in \Cl B(Q)^{A^1}$, we denote by $F\times G$
the function $A^1\times A^1\to\Cl B(Q)$ defined by $(F\times
G)(a,b)=F(a)G(b)$.

\begin{Def}
  \label{def:Rk-triple-expansion}
  Let $R^\omega(Q,A)$ denote the set of all triples $(F,B,u)$ such that
  $F\in\Cl B(Q)^{A^1}$, $B\in\Cl P(A)$, $u\in\om A{LRB}$, $F(a)=1$ for
  all $a\in A^1\setminus c(u)$, and $B\subseteq c(u)$. For two elements
  $(F,B,u)$ and $(G,C,v)$ of $R^\omega(Q,A)$, we define their product to
  be
  $$(F,B,u)(G,C,v)=\bigl((F\times G)\circ\chi^B_{u,v},D,uv\bigr),$$
  where the product $(D,uv)=(B,u)(C,v)$ is computed in~$\Cl L_A$.
\end{Def}

The triples in Definition~\ref{def:Rk-triple-expansion} provide a
refined model of pseudowords over~\pv R, where we add a first component
to the two that were previously considered. The underlying idea is to
capture the behavior on a finite automaton of the suffix of a pseudoword
starting with the first occurrence of a given letter.

The following result is a first requirement for the above definition to be
a good choice.

\begin{Prop}
  \label{p:Rk-semigroup}
  The set $R^\omega(Q,A)$ is a semigroup for the above multiplication.
\end{Prop}

\begin{proof}
  In view of Lemma~\ref{l:cumul-cont,LRB}, associativity is expressed by
  the formula
  $$\Bigl(\bigl((F\times G)\circ\chi^B_{u,v}\bigr)\times H\Bigr)
  \circ\chi^X_{uv,w} %
  =\Bigl(F\times \bigl((G\times H)\circ\chi^C_{v,w}\bigr)\Bigr)
  \circ\chi^B_{u,vw},$$
  where $(X,uv)=(B,u)(C,v)$. Under the natural extension of the notation
  $F\times G$ to three factors, the above equality may be rewritten as
  $$(F\times G\times H)
  \circ (\chi^B_{u,v},\mathrm{Id}_{A^1})\circ\chi^X_{uv,w} =(F\times
  G\times H) \circ(\mathrm{Id}_{A^1},\chi^C_{v,w})\circ\chi^B_{u,vw}.$$
  The proposition now follows from Lemma~\ref{l:over-associativity}.
\end{proof}

The following result amounts to a simple calculation in the semigroup
$R^\omega(Q,A)$.

\begin{Lemma}\label{l:natural-omeqa-power}
  For an arbitrary element $(F,B,u)$ of~$R^\omega(Q,A)$, its natural
  $\omega$-power is given by $(F,B,u)^\omega=(F_\omega,B,u)$, where
  $$F_\omega(a)=
  \begin{cases}
    1 &\text{if } a\in A^1\setminus c(u) \\
    F(a)F(\mathrm{i}_B(u))^{\omega-1} &\text{if } a\in c(u).
  \end{cases}
  $$
\end{Lemma}

\begin{proof}
  One can easily show by induction on~$n$ that, for $n>1$, we have
  $(F,B,u)^n=(F_n,B,u)$, where
  $$F_n(a)=
  \begin{cases}
    1 &\text{if } a\in A^1\setminus c(u) \\
    F(a)\bigl(F(\mathrm{i}_B(u))\bigr)^{n-1} &\text{if } a\in c(u).\popQED
  \end{cases}
  $$
\end{proof}

In case the base of the $\omega$-power is given as the product of two
elements of~$R^\omega(Q,A)$, the formula becomes somewhat more
complicated. We only sketch the routine proof, leaving the details to
the reader.

\begin{Lemma}\label{l:natural-omega-power-product}
  For arbitrary elements $(F,B,u)$ and $(G,C,v)$ of~$R^\omega(Q,A)$, the
  natural $\omega$-power of their product is given by
  $\bigl((F,B,u)(G,C,v)\bigr)^\omega=(H,D,uv)$, where $D=B$ if
  $c(v)\subseteq B$ while $D=C$ otherwise, and
  $$H(a)=
  \begin{cases}
    1 %
    & \text{if } %
    a\in A^1\setminus c(uv) \\
    \lefteqn{F(a)G(\mathrm{i}_B(v))
    \bigl(F(\mathrm{i}_D(u))G(\mathrm{i}_B(v))\bigr)^{\omega-1}}\\ %
    & \text{if } %
    a\in c(u) \wedge %
    \bigl(c(v)\subseteq B \vee %
    \mathrm{i}_D(u)\in c(u)\bigr) \\
    \lefteqn{F(a)G(\mathrm{i}_B(v))G(\mathrm{i}_C(v))^{\omega-1}}\\ %
    & \text{if } %
    a\in c(u)\wedge %
    \mathrm{i}_C(uv)\notin c(u) \wedge %
    c(v)\nsubseteq B \\
    \lefteqn{G(a)\bigl(F(\mathrm{i}_C(u))G(\mathrm{i}_B(v))\bigr)^{\omega-1}} \\ %
    & \text{if } %
    a\in c(v)\setminus c(u) \wedge %
    \mathrm{i}_C(u)\in c(u) \\
    G(a)G(\mathrm{i}_C(v))^{\omega-1} %
    & \text{if } %
    a\in c(v)\setminus c(u) \wedge %
    \mathrm{i}_C(uv)\notin c(u).
  \end{cases}
  $$
\end{Lemma}

\begin{proof}
  Taking into account that $\mathrm{i}_C(u)\in c(u)$ if and only if
  $\mathrm{i}_C(uv)\in c(u)$, in which case
  $\mathrm{i}_C(u)=\mathrm{i}_C(uv)$, it is easy to check that the
  conditions defining each case in the expression for $H(a)$ given in
  the statement of the lemma are mutually exclusive and cover all
  possibilities. It requires then only a simple calculation using
  Lemma~\ref{l:natural-omeqa-power} to verify that the values of $H(a)$
  are correctly given in each case.
\end{proof}

\section{An alternative \texorpdfstring{$\omega$}{w}-power}
\label{sec:alt-w-power}

Consider next an $A$-labeled digraph $\Cl G=(Q,A,\delta)$, with finite
set of vertices $Q$, and labeling given by a function %
$\delta:A\to\Cl B(Q)$, which is to be interpreted as meaning that there
is an edge $p\xrightarrow aq$ if and only if $(p,q)\in\delta(a)$. The
function $\delta$ determines a continuous homomorphism $(\Om AS)^1\to\Cl
B(Q)$, which is also denoted $\delta$. We also write $q\in pw$ to
indicate that $(p,q)\in\delta(w)$.

Given a subset $B$ of $A$, we let $\varepsilon(B)=\bigcup\delta(B^*)$;
in other words, a pair $(p,q)$ of elements of~$Q$ belongs to
$\varepsilon(B)$ if and only if there is some $w\in B^*$ such that $q\in
pw$. For $u\in(\Om AS)^1$, we also let $\varepsilon(u)=\varepsilon(c(u))$.

\begin{Def}
  \label{def:Rk,G}
  We associate with the finite $A$-labeled digraph $\Cl G=(Q,A,\delta)$
  an interpretation of the $\omega$-power in~$R^\omega(Q,A)$ as follows.
  For $(F,B,u)\in R^\omega(Q,A)$, let $(F_\omega,B,u)$ be the natural
  $\omega$-power of $(F,B,u)$ in the finite semigroup $R^\omega(Q,A)$.
  Then $(F,B,u)^{[\omega]}$ is defined to be the triple $(G,c(u),u)$,
  where $G(a)=F_\omega(a)\varepsilon(u)$ for each $a\in c(u)$ and
  $G(a)=1$ for all $a\in A\setminus c(u)$. This defines a unary
  semigroup structure on $R^\omega(Q,A)$, which depends on the choice of
  labeling $\delta$. We denote this unary semigroup $R^\omega(\Cl G)$.
\end{Def}

A word of warning is perhaps needed at this point. In a unary semigroup,
we most often use the notation $x\mapsto x^\omega$ to denote the unary
operation and we also use it for the abstract operation. However, in a
finite semigroup, the standard notation is to indicate $x^\omega$ as the
idempotent power of~$x$. Since, in the unary semigroup $R^\omega(\Cl
G)$, we consider a different unary operation, the notation
$x^{[\omega]}$ has been adopted. From hereon, we talk about
$\omega$-semigroups instead of unary semigroups.

For a triple $x$ in $R^\omega(Q,A)$, let $\pi_i(x)$ denote its $i$th
component. The following proposition shows that the $\omega$-semigroup
$R^\omega(\Cl G)$ has some nice properties.

\begin{Prop}
  \label{p:Rk,G-Rk-semigroup}
  For every finite $A$-labeled digraph $\Cl G=(Q,A,\delta)$, the
  $\omega$-semigroup $R^\omega(\Cl G)$ satisfies the following identities
  of $\omega$-semigroups:
  \begin{align*}
    &(x^{\omega})^{\omega}=(x^r)^{\omega}=x^{\omega}\quad (r\ge2),\\
    &(xy)^{\omega}x=(xy)^{\omega}x^{\omega}=(xy)^{\omega}.
  \end{align*}
\end{Prop}

\begin{proof}
  Let $(F',c(u),u)=x^{[\omega]}$. We first note that, from the
  definition of the multiplication it follows that
  $(F',c(u),u)(H,D,w)=(F',c(u),u)$ for every element $(H,D,w)$
  of~$R^\omega(\Cl G)$ such that $c(w)\subseteq c(u)$. In particular, we
  obtain the identities
  $(xy)^{[\omega]}x=(xy)^{[\omega]}x^{[\omega]}=(xy)^{[\omega]}$ and
  that $x^{[\omega]}$ is idempotent. Hence, for $a\in c(u)$,
  $\pi_1\bigl((x^{[\omega]})^{[\omega]}\bigr)(a)$ is
  $F'(a)\varepsilon(u)^2$ while it is 1 at $a\in A\setminus c(u)$. Since
  the relation $\varepsilon(u)$ is idempotent, it follows that
  $(x^{[\omega]})^{[\omega]}=x^{[\omega]}$. Finally, that
  $(x^r)^{[\omega]}=x^{[\omega]}$ follows from the fact
  $(x^r)^\omega=x^\omega$ in every finite semigroup.
\end{proof}

We now consider the subset $\tilde{R}^\omega(\Cl G)$ of
$R^\omega(Q,A)$ consisting of the triples $(F,B,u)$ such that
the following conditions hold for every $a\in c(u)$:
\begin{align}
  \label{eq:simple1}
  F(a)&\subseteq\varepsilon(u);
  \\
  \label{eq:simple2}
  F(a)&\varepsilon(B)=F(a).
\end{align}
Note that Property~\eqref{eq:simple1} implies that the inclusion
$F(a)\varepsilon(u)\subseteq\varepsilon(u)$ holds.

\begin{Lemma}
  \label{l:subsemigroup}
  The set $\tilde{R}^\omega(\Cl G)$ is a subsemigroup of $R^\omega(Q,A)$.
\end{Lemma}

\begin{proof}
  We verify only that Property~\eqref{eq:simple2} is preserved by
  multiplication, leaving it to the reader to verify that the same is
  true for Property~\eqref{eq:simple1}. Let $(F,B,u)$ and $(G,C,v)$ be
  arbitrary elements of $\tilde{R}^\omega(\Cl G)$ and consider the product
  $(H,D,uv)=(F,B,u)(G,C,v)$. We need to show that
  $H(a)\varepsilon(D)=H(a)$ for every $a\in c(uv)$.

  In case $c(v)\subseteq B$, we have $D=B$, $\mathrm{i}_B(v)=1$, and we
  may compute
  $$H(a)=F(a)G(\mathrm{i}_B(v))=F(a)
  =F(a)\varepsilon(B)=H(a)\varepsilon(D).$$
  Assume next that $c(v)\nsubseteq B$, so that $D=C$ and
  $\mathrm{i}_B(v)\in c(v)$. In case, additionally, $a\in c(u)$, we
  obtain
  $$H(a)=F(a)G(\mathrm{i}_B(v))
  =F(a)G(\mathrm{i}_B(v))\varepsilon(C)
  =H(a)\varepsilon(D).$$
  Finally, otherwise, that is when, additionally, $a\in c(uv)\setminus
  c(u)$, we get
  $$H(a)=G(a)=G(a)\varepsilon(C)=H(a)\varepsilon(D).\popQED$$
\end{proof}

Note that, for every $(F,B,u)\in\tilde{R}^\omega(\Cl G)$, its $\omega$-power
$(F,B,u)^{[\omega]}$ belongs to~$\tilde{R}^\omega(\Cl G)$. Hence,
$\tilde{R}^\omega(\Cl G)$~is in fact an $\omega$-subsemigroup of
$R^\omega(\Cl G)$. In particular, $\tilde{R}^\omega(\Cl G)$ satisfies all
the identities of Proposition~\ref{p:Rk,G-Rk-semigroup}.

\begin{Lemma}
  \label{l:general-fact}
  Let $\Cl G=(Q,A,\delta)$ be a finite $A$-labeled digraph, $B$ a subset
  of~$A$, and $s,t\in\Cl B(Q)$ be relations contained
  in~$\varepsilon(B)$. Then, the following equality holds:
  $(st)^\omega s\,\varepsilon(B)=(st)^\omega\varepsilon(B)$.
\end{Lemma}

\begin{proof}
  We have already observed that the definition of $\varepsilon(B)$
  implies that $s\,\varepsilon(B)$ and $t\,\varepsilon(B)$ are both
  contained in~$\varepsilon(B)$. Hence, the relation %
  $(st)^\omega s\,\varepsilon(B)$ is certainly contained
  in~$(st)^\omega\varepsilon(B)$. The reverse inclusion is obtained by
  noting that $(st)^\omega\varepsilon(B) %
  =(st)^\omega s \cdot t(st)^{\omega-1}\varepsilon(B) %
  \subseteq (st)^\omega s\,\varepsilon(B)$.
\end{proof}

While the $\omega$-semigroup $R^\omega(\Cl G)$ in general fails the
identity $(xy)^\omega=x(yx)^\omega$, it turns out that the
$\omega$-subsemigroup $\tilde{R}^\omega(\Cl G)$ does satisfy it.

\begin{Prop}
  \label{p:problematic-identity}
  The $\omega$-semigroup $\tilde{R}^\omega(\Cl G)$ satisfies the
  identity $(xy)^\omega=x(yx)^\omega$.
\end{Prop}

\begin{proof}
  Let $(F,B,u)$ and $(G,C,v)$ be arbitrary elements
  of~$\tilde{R}^\omega(\Cl G)$ and consider the corresponding
  expressions
  \begin{align*}
    (\tilde{H},c(uv),uv) &= \bigl((F,B,u)(G,C,v)\bigr)^{[\omega]} \\
    (\tilde{I},c(uv),vu) &= \bigl((G,C,v)(F,B,u)\bigr)^{[\omega]} \\
    (J,c(uv),uv) &= (F,B,u)\bigl((G,C,v)(F,B,u)\bigr)^{[\omega]}.
  \end{align*}
  Then, taking into account Lemma~\ref{l:natural-omega-power-product},
  we may compute
  $$\tilde{H}(a)=
  \begin{cases}
    1 %
    & \text{if } %
    a\in A^1\setminus c(uv) \\
    \lefteqn{F(a)G(\mathrm{i}_B(v))
      \bigl(F(\mathrm{i}_D(u))G(\mathrm{i}_B(v))\bigr)^{\omega-1} %
      \varepsilon(uv)}\\ %
    & \text{if } %
    a\in c(u) \wedge %
    \bigl(c(v)\subseteq B \vee %
    \mathrm{i}_D(u)\in c(u)\bigr) \\
    \lefteqn{F(a)G(\mathrm{i}_B(v))G(\mathrm{i}_C(v))^{\omega-1} %
      \varepsilon(uv)}\\ %
    & \text{if } %
    a\in c(u)\wedge %
    \mathrm{i}_C(uv)\notin c(u) \wedge %
    c(v)\nsubseteq B \\
    \lefteqn{G(a)\bigl(F(\mathrm{i}_C(u))G(\mathrm{i}_B(v))\bigr)^{\omega-1} %
      \varepsilon(uv)}\\ %
    & \text{if } %
    a\in c(v)\setminus c(u) \wedge %
    \mathrm{i}_C(u)\in c(u) \\
    G(a)G(\mathrm{i}_C(v))^{\omega-1} %
    \varepsilon(uv) %
    & \text{if } %
    a\in c(v)\setminus c(u) \wedge %
    \mathrm{i}_C(uv)\notin c(u)
  \end{cases}
  $$
  and, dually,
  $$\tilde{I}(a)=
  \begin{cases}
    1 %\hspace*{2cm} %
    & \text{if } %
    a\in A^1\setminus c(vu) \\
    \lefteqn{G(a)F(\mathrm{i}_C(u))
    \bigl(G(\mathrm{i}_E(v))F(\mathrm{i}_C(u))\bigr)^{\omega-1} %
    \varepsilon(uv)}\\ %
    & \text{if } %
    a\in c(v) \wedge %
    \bigl(c(u)\subseteq C \vee %
    \mathrm{i}_E(v)\in c(v)\bigr) \\
    \lefteqn{G(a)F(\mathrm{i}_C(u))F(\mathrm{i}_B(u))^{\omega-1} %
    \varepsilon(uv)}\\ %
    & \text{if } %
    a\in c(v)\wedge %
    \mathrm{i}_B(vu)\notin c(v) \wedge %
    c(u)\nsubseteq C \\
    \lefteqn{F(a)\bigl(G(\mathrm{i}_B(v))F(\mathrm{i}_C(u))\bigr)^{\omega-1} %
    \varepsilon(uv)}\\ %
    & \text{if } %
    a\in c(u)\setminus c(v) \wedge %
    \mathrm{i}_B(v)\in c(v) \\
    %\lefteqn{
      F(a)F(\mathrm{i}_B(u))^{\omega-1} %
    \varepsilon(uv) %}\\ %
    & \text{if } %
    a\in c(u)\setminus c(v) \wedge %
    \mathrm{i}_B(vu)\notin c(v)
  \end{cases}
  $$
  from which it follows that
  $$J(a)=
  \begin{cases}
    1 %
    & \text{if } %
    a\in A^1\setminus c(uv) \\
    F(a) %
    & \text{if } %
    a\in c(uv)=B %
    \\
    \lefteqn{F(a)G(\mathrm{i}_B(v))F(\mathrm{i}_C(u))
    \bigl(G(\mathrm{i}_E(v))F(\mathrm{i}_C(u))\bigr)^{\omega-1} %
    \varepsilon(uv)}\\ %
    & \text{if } %
    a\in c(u) \wedge c(v)\nsubseteq B
    \\
    \lefteqn{F(a)F(\mathrm{i}_B(u))^\omega %
    \varepsilon(uv)} \\ %
    & \text{if } %
    a\in c(u)\wedge\mathrm{i}_B(vu)\in c(u)\setminus c(v)
    \\
    \lefteqn{G(a)F(\mathrm{i}_C(u))
    \bigl(G(\mathrm{i}_E(v))F(\mathrm{i}_C(u))\bigr)^{\omega-1} %
    \varepsilon(uv)} \\ %
    & \text{if } %
    a\in c(v)\setminus c(u) \wedge %
    \bigl(c(u)\subseteq C \vee %
    \mathrm{i}_E(v)\in c(v)\bigr) % \\
  \end{cases}
  $$
  It remains to show that $J(a)=\tilde{H}(a)$ for every $a\in A$. We
  test the equality following the separation in cases in the above
  description of $J$.

  \smallskip
  
  \emph{Case 1.} In case $a\in A^1\setminus c(uv)$, we get
  $J(a)=1=\tilde{H}(a)$.

  \smallskip
  
  \emph{Case 2.} Suppose now that $a\in c(uv)=B$. Since $B\subseteq
  c(u)$, it follows that $c(v)\subseteq B=c(u)$, which yields $D=B$ and
  $\mathrm{i}_D(u)=\mathrm{i}_B(v)=1$. Hence, the only possible
  alternative in the above description of~$\tilde{H}$ is the second one.
  Moreover, it gives $\tilde{H}(a)=F(a)\varepsilon(uv)$. Since $(F,B,u)$
  belongs to~$\tilde{R}^\omega(\Cl G)$ and $B=c(uv)$, we do have
  $\tilde{H}(a)=F(a)\varepsilon(uv)=F(a)=J(a)$.

  \smallskip
  
  \emph{Case 3.} Suppose next that $a\in c(u)$ and $c(v)\nsubseteq B$.
  The latter assumption implies that $D=C$ and
  $\mathrm{i}_B(vu)=\mathrm{i}_B(v)$. There are now two possibilities,
  the first of which is to fall in Case~2 of~$\tilde{H}$, with
  $\mathrm{i}_D(u)\in c(u)$, that is, $c(u)\nsubseteq D=C$, which
  entails $E=B$. In this case, we obtain
  $J(a)=F(a)G(\mathrm{i}_B(v))F(\mathrm{i}_C(u)) %
  \bigl(G(\mathrm{i}_B(v))F(\mathrm{i}_C(u))\bigr)^{\omega-1}%
  \varepsilon(uv) $ %
  while %
  $\tilde{H}(a)=F(a)G(\mathrm{i}_B(v)) %
  \bigl(F(\mathrm{i}_C(u))G(\mathrm{i}_B(v))\bigr)^{\omega-1} %
  \varepsilon(uv) $ %
  so that the equality $J(a)=\tilde{H}(a)$ follows from
  Lemma~\ref{l:general-fact}. Alternatively, we fall in Case~3
  of~$\tilde{H}$, with $\mathrm{i}_C(uv)\notin c(u)$, which yields
  $c(u)\subseteq C$, so that $E=C$, and $\mathrm{i}_C(u)=1$. Hence, we
  obtain directly
  $J(a)=F(a)G(\mathrm{i}_B(v))G(\mathrm{i}_C(v))^{\omega-1} %
  \varepsilon(uv) %
  =\tilde{H}(a)$.

  \smallskip
  
  \emph{Case 4.} Assume now that $a\in c(u)$ and $\mathrm{i}_B(vu)\in
  c(u)\setminus c(v)$. The second condition implies that
  $\mathrm{i}_B(v)=1$, that is, $c(v)\subseteq B$, whence $D=B$, and
  $\mathrm{i}_B(vu)=\mathrm{i}_B(u)$. This means that we are in Case~2
  of~$\tilde{H}$ and we obtain %
  $\tilde{H}(a)=F(a)F(\mathrm{i}_B(u))^{\omega-1}\varepsilon(uv)$ while
  $\tilde{H}(a)=F(a)F(\mathrm{i}_B(u))^\omega\varepsilon(uv)$ and so the
  equality $\tilde{H}(a)=J(a)$ follows from Lemma~\ref{l:general-fact}.

  \smallskip
  
  \emph{Case 5a.} Here, we consider the case where $a\in c(v)\setminus
  c(u)$ and $c(u)\subseteq C$. The latter condition means that
  $\mathrm{i}_C(u)=1$ and implies that $E=C$. This falls in Case~5
  of~$\tilde{H}$ and we obtain %
  $\tilde{H}(a)=G(a)G(\mathrm{i}_C(v))^{\omega-1} %
  \varepsilon(uv) %
  =J(a)$.
  
  \smallskip
  
  \emph{Case 5b.} Assume finally that $a\in c(v)\setminus c(u)$,
  $c(u)\nsubseteq C$, and $\mathrm{i}_E(v)\in c(v)$. Since
  $c(u)\nsubseteq C$, we have $E=B$. Taking into account that
  $\mathrm{i}_E(v)\in c(v)$, we deduce that $D=C$. We fall in Case~4
  of~$\tilde{H}$, which gives the equality %
  $\tilde{H}(a)=G(a)\bigl(F(\mathrm{i}_C(u))G(\mathrm{i}_B(v))\bigr)^{\omega-1} %
  \varepsilon(uv)$ %
  while Case~5 of~$J$ provides the formula %
  $J(a)=G(a)F(\mathrm{i}_C(u)) %
  \bigl(G(\mathrm{i}_B(v))F(\mathrm{i}_C(u))\bigr)^{\omega-1} %
  \varepsilon(uv)$. %
  Applying Lemma~\ref{l:general-fact}, we conclude that
  $\tilde{H}(a)=J(a)$.
\end{proof}

Combining Proposition~\ref{p:problematic-identity} with
Proposition~\ref{p:Rk,G-Rk-semigroup}, we are led to the following key
result.

\begin{Prop}
  \label{p:Rk,G-Rk-semigroup-really!}
  The $\omega$-semigroup $\tilde{R}^\omega(\Cl G)$ belongs to the
  variety $\pv R^\omega$.
\end{Prop}

\begin{proof}
  It remains to invoke the result
  from~\cite[Theorem~6.1]{Almeida&Zeitoun:2003b} that the identities in
  Propositions~\ref{p:Rk,G-Rk-semigroup}
  and~\ref{p:problematic-identity} define the variety $\pv R^\omega$.
\end{proof}

We introduce a further restriction on the elements
of~$\tilde{R}^\omega(\Cl G)$, namely we consider the subset $S^\omega(\Cl
G)$ consisting of the elements $(F,B,u)$ of~$\tilde{R}^\omega(\Cl G)$ such
that
\begin{align}
  \label{eq:elusive2}
  X\subseteq Y\subseteq A &\implies
  F(\mathrm{i}_X(u))\subseteq\varepsilon(Y)F(\mathrm{i}_Y(u)).
\end{align}

\begin{Prop}
  \label{p:elusive}
  The set $S^\omega(\Cl G)$ is an $\omega$-subsemigroup of
  $R^\omega(\Cl G)$.
\end{Prop}

\begin{proof}
  Consider two elements $(F,B,u)$ and $(G,C,v)$ of~$R^\omega(Q,A)$ and
  their product $(H,D,uv)=(F,B,u)(G,C,v)$.
  
  Suppose that $(F,B,u)$ and $(G,C,v)$ satisfy
  Property~\eqref{eq:elusive2}. We claim that so does their product
  $(H,D,uv)$.

  Consider subsets $X$ and $Y$ of~$A$ such that $X\subseteq Y$. Assume
  first that $c(v)\subseteq B$, so that $c(v)\subseteq c(u)$,
  $\mathrm{i}_Y(uv)=\mathrm{i}_Y(u)$, and
  $$H(\mathrm{i}_X(uv))=F(\mathrm{i}_X(uv))G(\mathrm{i}_B(v)) %
  \text{ and } %
  H(\mathrm{i}_Y(uv))=F(\mathrm{i}_Y(uv))G(\mathrm{i}_B(v)).$$
  If $\mathrm{i}_X(uv)\notin c(u)$, then $\mathrm{i}_Y(uv)\notin c(u)$
  must also hold and $F(\mathrm{i}_X(uv))=1=F(\mathrm{i}_Y(uv))$, whence
  $H(\mathrm{i}_X(uv))=H(\mathrm{i}_Y(uv))
  \subseteq\varepsilon(Y)H(\mathrm{i}_Y(uv))$. %
  On the other hand, if $\mathrm{i}_X(uv)\in c(u)$, then we have
  $\mathrm{i}_X(uv)=\mathrm{i}_X(u)$.
  Hence, we may apply the assumption that $(F,B,u)$ satisfies
  Property~\eqref{eq:elusive2} to deduce that
  $$H(\mathrm{i}_X(uv)) %
  =F(\mathrm{i}_X(u))G(\mathrm{i}_B(v)) %
  \subseteq\varepsilon(Y)F(\mathrm{i}_Y(u))G(\mathrm{i}_B(v)) %
  =\varepsilon(Y)H(\mathrm{i}_Y(uv)).$$
  
  Assume next that $c(v)\nsubseteq B$. In case $\mathrm{i}_X(uv)\notin
  c(u)$, then $\mathrm{i}_Y(uv)\notin c(u)$ also holds, and we obtain
  $$H(\mathrm{i}_X(uv)) %
  =G(\mathrm{i}_X(v)) %
  \subseteq\varepsilon(Y)G(\mathrm{i}_Y(v)) %
  =\varepsilon(Y)H(\mathrm{i}_Y(uv)).$$
  We may, therefore, assume that $\mathrm{i}_X(uv)\in c(u)$. The
  additional assumption that $\mathrm{i}_Y(uv)\notin c(u)$ entails that
  $c(u)\subseteq Y$ and so $B\subseteq Y$, as $B\subseteq c(u)$. Since
  both $(F,B,u)$ and $(G,C,v)$ satisfy Property~\eqref{eq:elusive2}, we
  may deduce the following relations:
  \begin{align*}
    H(\mathrm{i}_X(uv)) %
    &=F(\mathrm{i}_X(u))G(\mathrm{i}_B(v)) %
    \subseteq\varepsilon(Y)F(\mathrm{i}_Y(u))G(\mathrm{i}_B(v)) %
    =\varepsilon(Y)G(\mathrm{i}_B(v)) \\
    &\subseteq\varepsilon(Y)\cdot\varepsilon(Y)G(\mathrm{i}_Y(v)) %
    =\varepsilon(Y)H(\mathrm{i}_Y(uv)).
  \end{align*}
  It remains to consider the case where both $\mathrm{i}_X(uv)$ and
  $\mathrm{i}_Y(uv)$ belong to~$c(u)$. Taking into account that
  $(F,B,u)$ satisfies Property~\eqref{eq:elusive2}, we obtain:
  $$H(\mathrm{i}_X(uv)) %
  =F(\mathrm{i}_X(u))G(\mathrm{i}_B(v)) %
  \subseteq\varepsilon(Y)F(\mathrm{i}_Y(u))G(\mathrm{i}_B(v)) %
  =\varepsilon(Y)H(\mathrm{i}_Y(uv)).$$

  To conclude the proof, we must show that the $\omega$-power
  $(I,c(u),u)=(F,B,u)^{[\omega]}$ satisfies Property~\eqref{eq:elusive2}
  if so does $(F,B,u)$. Let $(F_\omega,B,u)=(F,B,u)^\omega$, so that the
  function $I$ is given by the formula $I(a)=F_\omega(a)\varepsilon(u)$
  if $a\in c(u)$ and $I(a)=1$ otherwise. Let $X$ and $Y$ be such that
  $X\subseteq Y\subseteq A$. In case $c(u)\subseteq X$, we get
  $$I(\mathrm{i}_X(u)) %
  =1 %
  \subseteq\varepsilon(Y)1 %
  =\varepsilon(Y)I(\mathrm{i}_Y(u)).
  $$
  For the remainder of the proof, we assume that $c(u)\nsubseteq X$. From
  the assumption that $(F,B,u)$ satisfies Property~\eqref{eq:elusive2}
  and the previous step of the proof, we know that
  $(F_\omega,B,u)=(F,B,u)^\omega$, which is a finite power of~$(F,B,u)$,
  also satisfies Property~\eqref{eq:elusive2}. Hence, we obtain
  $$I(\mathrm{i}_X(u)) %
  =F_\omega(\mathrm{i}_X(u))\varepsilon(u) %
  \subseteq \varepsilon(Y)F_\omega(\mathrm{i}_Y(u))\varepsilon(u).$$
  In case $c(u)\nsubseteq Y$, the rightmost expression in the preceding
  inclusion is equal to $\varepsilon(Y)I(\mathrm{i}_Y(u))$. Otherwise,
  that expression reduces to $\varepsilon(Y)\varepsilon(u)$ and
  $$\varepsilon(Y)\varepsilon(u) %
  \subseteq\varepsilon(Y)\varepsilon(Y) %
  =\varepsilon(Y) %
  =\varepsilon(Y)I(\mathrm{i}_Y(u)),$$
  which concludes the proof.
\end{proof}

\section{A natural partial order and generators}
\label{sec:order}

Given two elements $x$ and $y$ of~$R^\omega(Q,A)$, we write $x\le y$ if
$\pi_1(x)\subseteq\pi_1(y)$, $\pi_2(x)\subseteq\pi_2(y)$, and
$\pi_3(x)=\pi_3(y)$. This defines a partial order on $R^\omega(Q,A)$.

\begin{Prop}
  \label{p:order-stability}
  The order $\le$ is stable under multiplication on the left. The
  restriction of the order $\le$ to $S^\omega(\Cl G)$ is stable
  under multiplication on the right.
\end{Prop}

\begin{proof}
  Let $(F,B,u)$, $(G,C,v)$, and $(H,D,w)$ be elements
  of~$R^\omega(Q,A)$.

  Suppose that the inequality $(F,B,u)\le(G,C,v)$ holds so that, in
  particular, we have $u=v$. Let
  $$(I,X,wu)=(H,D,w)(F,B,u) %
  \text{ and } %
  (J,Y,wu)=(H,D,w)(G,C,u).$$
  In case $c(u)\subseteq D$, we get $X=D=Y$ and
  $$I(a)=H(a)F(\mathrm{i}_D(u))\subseteq H(a)G(\mathrm{i}_D(u))=J(a);$$
  note that the conditions in the previous line also hold if $a\in
  c(w)$. We now assume that $c(u)\nsubseteq D$, which yields
  $X=B\subseteq C=Y$. It remains to consider the case where $a\notin
  c(w)$, in which we obtain $I(a)=F(a)\subseteq G(a)=J(a)$. This
  completes the proof of left stability.

  For the proof of right stability within $S^\omega(\Cl G)$, we assume
  that the triples $(F,B,u)$, $(G,C,v)$, and $(H,D,w)$ are elements
  of~$S^\omega(\Cl G)$ such that the inequality $(F,B,u)\le(G,C,v)$
  holds, so that $u=v$. Consider the products
  $$(I,X,uw)=(F,B,u)(H,D,w)
  \text{ and } %
  (J,Y,uw)=(G,C,u)(H,D,w).$$
  Suppose first that $c(w)\subseteq B$, whence also $c(w)\subseteq C$
  holds. It follows that $X=B\subseteq C=Y$ and
  $$I(a) %
  =F(a)H(\mathrm{i}_B(w)) %
  =F(a) %
  \subseteq G(a) %
  =G(a)H(\mathrm{i}_C(w)) %
  =J(a).$$
  From hereon, we suppose that $c(w)\nsubseteq B$. In case
  $c(w)\subseteq C$, we get $X=D\subseteq c(w)\subseteq C=Y$. In case
  $c(w)\nsubseteq C$, we obtain $c(w)\nsubseteq B$ and $X=D=Y$.

  Next, we assume that $a\in c(u)$, so that
  $$I(a) %
  =F(a)H(\mathrm{i}_B(w)) %
  \subseteq G(a)\varepsilon(C)H(\mathrm{i}_C(w)) %
  =G(a)H(\mathrm{i}_C(w)) %
  =J(a),$$
  where the inclusion uses the inequality $F\subseteq G$ and the
  assumption that $(H,D,w)$ satisfies Property~\eqref{eq:elusive2}, and
  the second equality comes from the hypothesis that $(G,C,v)$ satisfies
  Property~\eqref{eq:simple2}.

  Finally, consider the case where $a\notin c(u)$. In case
  $c(w)\subseteq C$, since $c(w)\subseteq C\subseteq c(u)$, we get
  $I(a)=H(a)=1$, while we also have $J(a)=G(a)H(\mathrm{i}_C(w))=1$.
  Otherwise, that is in the case where $c(w)\nsubseteq C$, we simply get
  $I(a)=H(a)=J(a)$. This concludes the proof of right stability.
\end{proof}

Let $\Cl T^\omega_A$ denote the algebra of $\omega$-terms over $A$, that
is, the unary algebra freely generated by~$A$, in which the unary
operation is represented by the $\omega$-power.

Next, we choose special elements in~$R^\omega(Q,A)$.

\begin{Def}
  \label{def:RkG}
  Consider a finite $A$-labeled digraph $\Cl G=(Q,A,\delta)$. For each
  letter $a\in A$, let the triple
  $\nu_{[\omega]}(a)=(F_a,\emptyset,a)$ be determined by
  $$F_a(b)=
  \begin{cases}
    \delta(a)&\mbox{if } b=a,\\
    1&\mbox{otherwise}.
  \end{cases}
  $$
  Note that $(F_a,\emptyset,a)$ belongs to~$S^\omega(\Cl G)$. We define
  two homomorphisms $\Cl T^\omega_A\to S^\omega(\Cl G)$ of
  $\omega$-semigroups by letting
  $\nu_{\omega}(a)=\nu_{[\omega]}(a)=(F_a,\emptyset,a)$ for each $a\in
  A$: for $\nu_{\omega}$, we consider the natural structure of
  $\omega$-semigroup of~$S^\omega(\Cl G)$ while, for $\nu_{[\omega{}]}$,
  we take its alternative $\omega$-power defined in
  Section~\ref{sec:alt-w-power}.
\end{Def}

The unique homomorphism of $\omega$-semigroups $\Cl T^\omega_A\to\omo
AS$ mapping each generator $a\in A$ to itself is denoted $\eta$. In view
of Proposition~\ref{p:Rk,G-Rk-semigroup-really!}, we may consider the
unique homomorphism of $\omega$-semi\-groups $\rho_\Cl G:\omo AR\to
S^\omega(\Cl G)$ that maps each generator $a\in A$ to the
triple~$(F_a,\emptyset,a)$.

The following result further explains our choice of multiplication
in~$\Cl L_A$.

\begin{Lemma}
  \label{l:contents}
  For each $\alpha\in\Cl T^\omega_A$, the following properties hold:
  \begin{enumerate}[label=$(\roman*)$,ref=$(\roman*)$,leftmargin=*]
  \item\label{item:content}
    $c(\pi_3(\nu_{[\omega]}(\alpha)))=c(\eta(\alpha))$;
  \item\label{item:cumulative-content}
    $\pi_2(\nu_{[\omega]}(\alpha))=\vec c(\eta(\alpha))$.
  \end{enumerate}
\end{Lemma}

\begin{proof}
  The proof is done by induction on the construction of the
  $\omega$-term $\alpha$. If $\alpha$~is a letter, then the result is
  obtained by direct inspection. Assuming that $\alpha=\beta\gamma$, the
  definitions and the induction hypothesis for both $\beta$ and $\gamma$
  yield
  \begin{align*}
    c(\pi_3(\nu_{[\omega]}(\alpha))) %
    &=c(\pi_3(\nu_{[\omega]}(\beta\gamma))) %
    =c(\pi_3(\nu_{[\omega]}(\beta)))\cup c(\pi_3(\nu_{[\omega]}(\gamma))) \\
    &=c(\eta(\beta))\cup c(\eta(\gamma)) %
    =c(\eta(\alpha)).
  \end{align*}
  Similarly, since $\vec c(\eta(\alpha))$ is equal to $\vec
  c(\eta(\beta))=\pi_2(\nu_{[\omega]}(\beta))$ if
  $c(\eta(\gamma))\subseteq\vec c(\eta(\beta))$, and to~$\vec
  c(\eta(\gamma))=\pi_2(\nu_{[\omega]}(\gamma))$ otherwise, we
  get $\vec c(\eta(\alpha))=\pi_2(\nu_{[\omega]}(\alpha))$ by
  definition of the multiplication in~$\Cl L_A$.

  Suppose next that the induction hypothesis holds for the $\omega$-term
  $\alpha$. By definition of the $[\omega]$-power and since
  $\nu_{[\omega]}(\alpha^\omega)=\nu_{[\omega]}(\alpha)^{[\omega]}$, we
  must have $\pi_2(\nu_{[\omega]}(\alpha^\omega)) %
  =c(\pi_3(\nu_{[\omega]}(\alpha^\omega))) %
  =c(\pi_3(\nu_{[\omega]}(\alpha)))$. As $\alpha$
  satisfies~\ref{item:content} and $\vec
  c(\eta(\alpha^\omega))=c(\eta(\alpha))$, we deduce that
  $\alpha^\omega$ still satisfies both \ref{item:content}
  and~\ref{item:cumulative-content}.
\end{proof}

In particular, we obtain the following result.

\begin{Prop}
  \label{p:idempotency}
  An $\omega$-term $\alpha\in\Cl T^\omega_A$ is such that $\pj
  R(\eta(\alpha))$ is idempotent if and only if the equality
  $\pi_2(\nu_{[\omega]}(\alpha)) %
  =c(\pi_3(\nu_{[\omega]}(\alpha)))$ holds.\qed
\end{Prop}

Further properties of the order relation are established in the
following lemma.

\begin{Lemma}
  \label{l:order}
  For a finite $A$-labeled digraph \Cl G, the following conditions hold
  for all elements $x$ and $y$ of the $\omega$-semigroup $S^\omega(\Cl
  G)$, every $\omega$-term $\alpha\in\Cl T^\omega_A$, and every letter
  $a\in A$:
  \begin{enumerate}[label=$(\roman*)$,ref=$(\roman*)$,leftmargin=*]
  \item\label{item:omega-rise} $x^\omega\le x^{[\omega]}$;
  \item\label{item:omega-stability} %
    $x\le y$ implies $x^{[\omega]}\le y^{[\omega]}$;
  \item\label{item:omega-rise-general} %
    $\nu_\omega(\alpha)\le\nu_{[\omega]}(\alpha)$.
  \end{enumerate}
\end{Lemma}

\begin{proof}
  Properties \ref{item:omega-rise} and \ref{item:omega-stability} follow
  immediately from the interpretation of the $\omega$-power given by
  Definition~\ref{def:Rk,G}. Property~\ref{item:omega-rise-general} can
  then easily be deduced by induction on the construction of the
  $\omega$-term $\alpha$ in terms of the application of the operations
  of multiplication and $\omega$-power taking into account that the
  order $\le$ in $S^\omega(\Cl G)$ is stable under multiplication
  by Proposition~\ref{p:order-stability}.
\end{proof}

Following standard terminology, we say that $(S,{\le})$ is an
  \emph{ordered $\omega$-semi\-group} if $S$ is an $\omega$-semigroup and 
  $\le$ is a partial order on~$S$ which is compatible with multiplication
  and $\omega$-power.

\begin{Prop}
  \label{p:ordered-w-semigroup}
  The pair $(S^\omega(\Cl G),{\le})$ is an ordered $\omega$-semigroup.
\end{Prop}

\begin{proof}
  The order is compatible with multiplication by
  Proposition~\ref{p:order-stability} and with $\omega$-power by
  Lemma~\ref{l:order}\ref{item:omega-stability}.
\end{proof}

\section{Recognition of \texorpdfstring{\pv R}{R}-closures}
\label{sec:recognition}

Given two $\omega$-semigroups $S$ and $T$, a \emph{relational morphism}
$S\to T$ is a binary relation $\mu\subseteq S\times T$ with domain $S$
which is an $\omega$-subsemigroup of~$S\times T$. For a finite
$A$-labeled digraph $\Cl G=(Q,A,\delta)$, let $S(\Cl G)$ be the
semigroup $\delta(A^+)$. The relational morphism of
$\omega$-semigroups %
$\mu_\Cl G:S(\Cl G)\to S^\omega(\Cl G)$ is the composite relation
$\rho_\Cl G\circ\pj R\circ(\delta|_{\omo AS})^{-1}$.

Note that the composite mapping $\tilde\rho_\Cl G=\rho_\Cl G\circ\pj R$
is a homomorphism of $\omega$-semigroups. On the other hand, the
restriction of $\tilde\rho_\Cl G$ to~$A$ extends to a continuous
homomorphism $\Om AS\to S^\omega(\Cl G)$. Its restriction to \omo AS is
denoted~$\zeta$. For an $\omega$-term $\alpha\in\Cl T^\omega_A$ on the
alphabet $A$ representing the $\omega$-word $w\in\omo AS$, note that
$\nu_\omega(\alpha)=\zeta(w)$ and %
$\nu_{[\omega]}(\alpha)=\tilde\rho_\Cl G(w)$. %

Denote by $S_A^\omega(\Cl G)$ the subsemigroup of~$S^\omega(\Cl G)$
generated by~$\tilde\rho_\Cl G (A)$. Note that $S_A^\omega(\Cl G)$
consists of elements of 
$S^\omega(\Cl G)$ of the form $(F,\emptyset,u)$. There is another
mapping that plays a role in our construction. It is the mapping
$\xi:S_A^\omega(\Cl G)\to\Cl B(Q)$ which sends the triple $(F,B,u)$ to
the binary relation $F(\mathrm{i}_\emptyset(u))$. It follows from the definition
of the multiplication in~$R^\omega(Q,A)$ that the restriction
$\xi'=\xi|_{S_A^\omega(\Cl G)}$~is a homomorphism of semigroups, taking
its values in $S(\Cl G)$.

The relevant mappings are depicted in the following diagram:
$$
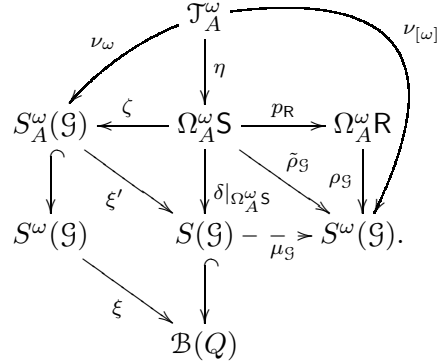
\xymatrix{
  & %
  \Cl T^\omega_A %
  \ar@/_2mm/[ld]_{\nu_\omega} %
  \ar[d]^\eta %
  \ar@/^18mm/[rdd]^{\nu_{[\omega]}} %
  & %
  \\
  S_A^\omega(\Cl G) %
  \ar[rd]_{\xi'} %
  \ar@{^{((}->}[d] %
  & %
  \omo AS %
  \ar[l]_\zeta %
  \ar[d]^(0.65){\delta|_{\omo AS}} %
  \ar[r]^{\pj R} %
  \ar[rd]^{\tilde\rho_\Cl G} & %
  \omo AR %
  \ar[d]_{\rho_\Cl G}\\
  S^\omega(\Cl G) %
  \ar[rd]_{\xi} %
  & %
  S(\Cl G) %
  \ar@{-->}[r]_{\mu_\Cl G} %
  \ar@{^{((}->}[d] %
  & %
  S^\omega(\Cl G). \\
  & %
  \Cl B(Q) %
  & %
}$$
Note that the diagram commutes. In view of
Lemma~\ref{l:order}\ref{item:omega-rise-general}, the inequality
$\zeta(w)\le\tilde\rho_\Cl G(w)$ holds for every $w\in\omo AS$.

Assuming that the language $L\subseteq A^+$ is recognized by some
automaton obtained from the $A$-labeled digraph \Cl G by adding an
appropriate choice of sets $I$ and $T$, respectively of initial and
terminal vertices, the language $L$~is also recognized by the transition
homomorphism $\delta|_{A^+}:A^+\to \Cl B(Q)$, namely
$$L=(\delta|_{A^+})^{-1}
\{\theta\in\Cl B(Q):\theta\cap(I\times T)\ne\emptyset\}.$$
It follows that the homomorphism $\delta|_{\omo AS}:\omo AS\to\Cl B(Q)$ recognizes
$\tclc L$ as
$$\tclc L=(\delta|_{\omo AS})^{-1}
\{\theta\in\Cl B(Q):\theta\cap(I\times T)\ne\emptyset\}.$$
whence so does $\zeta$.

\begin{Thm}
  \label{t:R-closures}
  Let $\Cl A=(Q,A,\delta,I,T)$ be a finite automaton and consider the
  underlying $A$-labeled digraph $\Cl G=(Q,A,\delta)$ and the set
  $$P=\{x\in S^\omega(\Cl G): %
  \xi(x)\cap (I\times T)\ne\emptyset\}.$$
  Then the equality $\rho_\Cl G^{-1}(P)=\tclcv RL$ holds for the
  language $L$ recognized by~\Cl A.
\end{Thm}

\begin{proof}
  Let $w$ be an arbitrary element of~$\tclcv RL$. Since \pv R is
  $\omega$-full \cite[Theorem~7.4]{Almeida&Costa&Zeitoun:2014}, there is
  some $v\in\tclc L$ such that $\pj R(v)=w$. Let $\alpha\in\Cl
  T^\omega_A$ be an $\omega$-term such that $\eta(\alpha)=v$ and let
  $v_n$ be the word that is obtained from $\alpha$ by replacing each
  subterm of the form $u^\omega$ by $u^{n!}$. Then $\lim v_n=v$. Since
  the closure of $L$ in~\Om AS is an open
  set~\cite[Theorem~3.6.1]{Almeida:1994a}, it follows that $v_n\in L$
  for all sufficiently large $n$. For each $n$, $\rho_\Cl G(v_n)$ is an
  element of~$S_A^\omega(\Cl G)$. Since $S^\omega(\Cl G)$~is finite, there
  is some subsequence $(v_{n_k})_k$ such that $s=\rho_\Cl G(v_{n_k})\in
  P$ is independent of~$k$. Note that the sequence $(\rho_\Cl G(v_n))_n$
  is eventually constant with value $\nu_\omega(\alpha)$. In particular,
  that value must be~$s$. On the other hand, %
  $\rho_\Cl G(w)=\tilde\rho_\Cl G(v)=\nu_{[\omega]}(\alpha)$. In view of
  Lemma~\ref{l:order}\ref{item:omega-rise-general}, it follows that
  $\rho_\Cl G(w)\in P$ since $P$~is upward closed with respect to the
  order~$\le$.

  For the reverse inclusion, let $w$ be an arbitrary element of
  $\tilde\rho_\Cl G^{-1}(P)$. We claim that there is $v\in\tclc L$ such
  that $\pj R(v)=\pj R(w)$, which shows that %
  $\rho_\Cl G^{-1}(P)\subseteq\tclcv RL$. Since $\zeta$ recognizes
  $\tclc L$, it suffices to show that there is $v\in\omo AS$ such that
  $\zeta(v)\in P$ and $\pj R(v)=\pj R(w)$. More generally, suppose that
  $\alpha\in\Cl T^\omega_A$ is such that the pair of states $(p,q)$
  belongs to~$\xi(\nu_{[\omega]}(\alpha))$. We claim that there is some
  $\beta\in\Cl T^\omega_A$ such that %
  $\pj R(\eta(\alpha))=\pj R(\eta(\beta))$ and %
  $(p,q)\in \xi(\nu_\omega(\beta))$. We prove the claim by induction on
  the construction of the $\omega$-term $\alpha$, in terms of the
  operations of multiplication and $\omega$-power.

  If $\alpha=\alpha_1\alpha_2$, then there exists $r\in Q$ such that
  $(p,r)\in\xi(\nu_{[\omega]}(\alpha_1))$ and
  $(r,q)\in\xi(\nu_{[\omega]}(\alpha_2))$. %
  Assuming the claim holds for both $\alpha_i$, there is %
  $\beta_i\in\Cl T^\omega_A$ such that %
  $\pj R(\eta(\alpha_i))=\pj R(\eta(\beta_i))$ ($i=1,2$), %
  $(p,r)\in \xi(\nu_\omega(\beta_1))$, and %
  $(r,q)\in \xi(\nu_\omega(\beta_2))$. %
  Then the $\omega$-term $\beta=\beta_1\beta_2$ has the required
  properties.

  Suppose next that $\alpha=\alpha_0^\omega$, where the claim holds for
  the $\omega$-term $\alpha_0$. By hypothesis, the pair of states
  $(p,q)$ belongs to the relation
  $\xi(\nu_{[\omega]}(\alpha))=\xi(\nu_{[\omega]}(\alpha_0)^{[\omega]})$.
  Let $n\ge1$ be such that $S^\omega(\Cl G)$ satisfies the identity
  $x^\omega=x^n$ for the natural $\omega$-power. In view of the
  definition of the $[\omega]$-power in %
  $S^\omega(\Cl G)$, it follows that there is some state $r\in Q$ such
  that $(p,r)\in\xi(\nu_{[\omega]}(\alpha_0^n))$, where
  $\nu_{[\omega]}(\alpha_0)^\omega %
  =\nu_{[\omega]}(\alpha_0)^n %
  =\nu_{[\omega]}(\alpha_0^n)$, %
  and there is a path in~\Cl G from $r$ to~$q$ labeled by some word $u$
  of~$c(\eta(\alpha_0))^*$. Since $\nu_{[\omega]}(\alpha_0^n)$ is
  idempotent, by the pigeonhole principle there is some state $r'\in Q$
  such that each of the pairs $(p,r')$, $(r',r')$, and $(r',r)$ belongs
  to~$\xi(\nu_{[\omega]}(\alpha_0^n))$. By the case of the product,
  already handled in the preceding paragraph, since $\alpha_0$ is
  assumed to satisfy the claim, so does $\alpha_0^n$. Hence, there are
  $\omega$-terms %
  $\beta_i\in\Cl T^\omega_A$ such that %
  $\pj R(\eta(\beta_i))=\pj R(\eta(\alpha_0^n))$ ($i=1,2,3$),
  $(p,r')\in\xi(\nu_\omega(\beta_1))$,
  $(r',r')\in\xi(\nu_\omega(\beta_2))$, and
  $(r',r)\in\xi(\nu_\omega(\beta_3))$. %
  Since elements of~\Om AS with the same image under \pj R have the same
  content, we obtain the equalities %
  $$\pj R(\eta(\beta_1\beta_2^\omega\beta_3 u)) %
  =\pj R(\eta(\beta_1\beta_2^\omega\beta_3)) %
  =\pj R(\eta(\alpha_0^\omega)) %
  =\pj R(\eta(\alpha)).$$
  We have thus shown that the $\omega$-term
  $\beta=\beta_1\beta_2^\omega\beta_3 u$ has all the required
  properties, thereby concluding the induction step and the proof of the
  theorem.
\end{proof}

Note that we may use the same labeled digraph to recognize several
languages.

\begin{Cor}
  \label{c:recognizing-several-closures}
  Let $L_1,\ldots,L_n$ be regular languages over the same finite
  alphabet $A$ and suppose that, for each $i\in\{1,\ldots,n\}$, a
  suitable choice of initial and terminal states in the $A$-labeled
  digraph $\Cl G=(Q,A,\delta)$ yields an automaton recognizing $L_i$.
  Then $S^\omega(\Cl G)$ recognizes every Boolean combination of the
  sets $\tclcv R{L_i}$.
\end{Cor}

\begin{proof}
  It suffices to apply Theorem~\ref{t:R-closures} and note that inverse
  functions behave well with respect to Boolean operations.
\end{proof}

For instance, one may take the same labeled digraph \Cl G to recognize
several given regular languages over the same finite alphabet $A$, such
as the disjoint union of their minimal automata where, naturally, the
choice of initial and terminal states depends on the language.
Corollary~\ref{c:recognizing-several-closures} then provides an
algorithm to compute the intersection of the pro-\pv R closures of the
given languages in~\omo AR.

Recall that a subset $P$ of a finite semigroup $S$ is said to be
\emph{\pv V-pointlike} if, for every relational morphism $\mu:S\to T$
into a semigroup from~\pv V, there is some $t\in T$ such that
$P\times\{t\}\subseteq\mu$. Equivalently, one may consider an arbitrary
onto homomorphism $\varphi:A^+\to S$, where $A$ is a finite alphabet,
and require that the closures of the regular languages $\varphi^{-1}(s)$
($s\in S$) in~\Om AV have some point in common. Since \pv R is
completely tame for the signature $\omega$, the preceding property for
the pseudovariety \pv R is equivalent to the sets $\tclcv
R{\varphi^{-1}(s)}$ ($s\in S$) having some point in common. In view of
Corollary~\ref{c:recognizing-several-closures}, we obtain the following
result.

\begin{Cor}
  \label{c:R-pointlikes}
  It is decidable whether a given subset of a finite semigroup is \pv
  R-pointlike.\qed
\end{Cor}

The previous corollary is not new. It was first proved
in~\cite{Almeida&Costa&Zeitoun:2004}. The proof may also be derived from
much more general results from~\cite{Almeida&Costa&Zeitoun:2005b} and
yet another approach to compute \pv R-pointlike sets was obtained
in~\cite{Almeida&Costa&Zeitoun:2006}.

We next consider a further algorithmic property associated with a
pseudovariety \pv V, which is important in the computation of Mal'cev
products with~\pv V. A subset $P$ of a finite semigroup $S$ is said to
be a \emph{\pv V-idempotent-pointlike subset} if, for every relational
morphism $\mu:S\to T$ into a semigroup from~\pv V, there is an
idempotent $e\in T$ such that $P\times\{e\}\subseteq\mu$.

\begin{Cor}
  \label{c:R-idempotent-pointlikes}
  It is decidable whether a given subset of a finite semigroup is \pv
  R-idempotent pointlike.
\end{Cor}

\begin{proof}
  Choose a finite alphabet $A$ and an onto homomorphism $\varphi:A^+\to S$,
  where $S$ is a given finite semigroup. Let $P$ be a finite subset
  of~$S$. By tameness \cite{Almeida&Costa&Zeitoun:2005b}, $P$ is \pv
  R-idempotent pointlike if and only if there is, for each $s\in P$,
  some $\omega$-word $w_s\in\tclcv S{\varphi^{-1}(s)}$ such that $\pj
  R(w_s)$ is the same idempotent independent of~$s$. Since %
  $\pj R(\tclcv S{\varphi^{-1}(s)}) %
  =\tclcv R{\varphi^{-1}(s)}$ %
  by fullness, we conclude that $P$ is \pv R-idempotent pointlike if and
  only if the intersection $I=\bigcap_{s\in P}\tclcv R{\varphi^{-1}(s)}$
  contains some idempotent.

  By Corollary~\ref{c:recognizing-several-closures}, there is a finite
  $A$-labeled digraph \Cl G such that $S^\omega(\Cl G)$ recognizes $I$.
  In view of Proposition~\ref{p:idempotency}, $I$ contains an idempotent
  if and only if the image of $I$ under~$\rho_{\Cl G}$ contains some
  triple $(F,B,u)$ such that $B=c(u)$, a condition that may be
  effectively tested.
\end{proof}

Again, Corollary~\ref{c:R-idempotent-pointlikes} follows from the
general tameness results of~\cite{Almeida&Costa&Zeitoun:2005b} but the
algorithms that may be derived from tameness are merely theoretical,
depending on enumerating in parallel all favorable and unfavorable
cases, until our instance of the problem is
produced~\cite{Almeida&Steinberg:2000a}. The algorithm described in the
proof of Corollary~\ref{c:R-idempotent-pointlikes} is much more
effective.

\begin{Prop}
  \label{p:size-of-RkG}
  For a given finite labeled digraph $\Cl G=(Q,A,\delta)$, let $m=|Q|$
  and $n=|A|$. Then, the cardinality of $S^\omega(\Cl G)$ is bounded
  above by $2^{(m^2+1)n}\cdot3\cdot n!$.
\end{Prop}

\begin{proof}
  Recall that the elements of~$S^\omega(\Cl G)$ are triples $(F,B,u)$
  where $F$ is a function $A^1\to\Cl B(Q)$, $B\subseteq A$, and $u\in\Om
  A{LRB}$. The cardinality $|\Om A{LRB}|$ is the number of different
  words in $n$~letters without repeated letters. It is well known to be
  equal to %
  $n!\sum_{r=0}^n\frac1{r!}=\lfloor e\cdot n!\rfloor$, whence it is
  bounded above by $3\cdot n!$. %
  For an element $(F,B,u)$ of~$S^\omega(\Cl G)$, the function
  $F:A^1\to\Cl B(Q)$ is such that $F(1)=1$. There are $2^{m^2n}$ such
  functions.
\end{proof}

A finer analysis taking into account Properties
\eqref{eq:simple1}--\eqref{eq:elusive2} may lead to better estimates.
Even better estimates may perhaps hold for the $\omega$-subsemi\-group
$\mathop{\mathrm{Im}}\tilde{\rho}_{\Cl G}$.

\section*{Acknowledgments}

This work was partly supported by the \textsc{Pessoa} French-Portuguese
project ``Separation in automata theory: algebraic, logical, and
combinatorial aspects''.
The work of the first two authors was also partially supported
respectively by CMUP (UID/MAT/00144/2019) and CMAT (UID/MAT/
00013/2013), which are funded by FCT (Portugal) with national (MCTES)
and European structural funds (FEDER), under the partnership agreement
PT2020.
The work of the third author was also partly supported by the DeLTA
project ANR-16-CE40-0007.

\bibliographystyle{amsplain}
\bibliography{R-closures}
\end{document}